\documentclass[12pt]{article}
\usepackage{stmaryrd}
\usepackage{bbm}
\usepackage[french,english]{babel}

\usepackage{setspace,mathtools,amsmath,amssymb,amsthm,fullpage,outline,centernot,slashed} 
\usepackage[pdftex]{graphicx}
\usepackage{mathrsfs} 
\usepackage{enumerate}
\usepackage{accents}
\usepackage[colorlinks=true,citecolor=black,linkcolor=black,urlcolor=blue]{hyperref}
\usepackage{upgreek}

\theoremstyle{plain}
\newtheorem{theorem}{Theorem}
\newtheorem{lemma}[theorem]{Lemma}
\newtheorem{corollary}[theorem]{Corollary}

\newtheorem{condition}[theorem]{Condition}
\newtheorem{claim}[theorem]{Claim}
\theoremstyle{definition}
\newtheorem{definition}[theorem]{Definition}

\newtheorem{conjecture}[theorem]{Conjecture}

\theoremstyle{remark}

\newcommand{\R}{\mathbb{R}}

\newcommand{\ben}{\begin{enumerate}}
\newcommand{\een}{\end{enumerate}}
\newcommand{\bit}{\begin{itemize}}
\newcommand{\eit}{\end{itemize}}

\def\bal#1\eal{\begin{align*}#1\end{align*}}

\DeclareMathSymbol{\mlq}{\mathord}{operators}{``}
\DeclareMathSymbol{\mrq}{\mathord}{operators}{`'}

\def\XXint#1#2#3{{\setbox0=\hbox{$#1{#2#3}{\int}$ }
\vcenter{\hbox{$#2#3$ }}\kern-.6\wd0}}

\title{Bounds on Growth and Impossibility of Collapse for Point Vortex Systems}
\author{Samuel Zbarsky}

\begin{document}
\maketitle
\begin{abstract}
We consider 2D point vortex systems and, under certain conditions on the masses of the point vortices, prove that collapse is impossible and provide bounds on the growth of the system. The bounds are typically of the form $O(t^a)$ for some $a<1/2$, but we obtain various results for various assumptions on the masses.
\end{abstract}

\section{Introduction}
We discuss point vortex systems in $\R^2$. This is a model for 2-dimensional incompressible fluids where vorticity is concentrated around certain points. The model is defined by having $n$ point vortices with nonzero signed masses $m_1,\ldots,m_n$ and trajectories $x_i(t)$ which satisfy the ODE
\[
\frac{d}{dt} x_i(t)=\sum_{j\ne i} m_j K(x_i,x_j).
\]
where
\[
K(x,y)=\frac{(x-y)^\perp}{|x-y|^2}.
\]
Here $(z_1,z_2)^\perp=(-z_2,z_1)$ and we assume that at $t=0$, no two point vortices are at the same point.

We are interested in two related questions:
\begin{enumerate}
\item Under what conditions on the masses can we say that the ODE has global existence?
\item For given masses, what bounds can we give on solutions to the ODE?
\end{enumerate}
We begin with a discussion of previous work on the behavior of point vortex systems. A more thorough introduction can be obtained by reading \cite{MarchioroPulvurenti94pointchapter} and \cite{Aref07}. First, we note that there are two other ways to think of point vortex systems. One is to identify $\R^2$ with the complex plane, and use
\begin{equation}\label{complexkernel}
K(z_1,z_2)=\frac{i}{\ \overline{z_1-z_2}\ }
\end{equation}
to rewrite the ODE. Another approach is to note that a point vortex system is a Hamiltonian system with the position and momentum variables being the $x$ and $y$ coordinates of vortices (suitably rescaled) and the Hamiltonian being
\[
H=\sum_{1\le i<j\le n} m_im_j\log|x_i-x_j|.
\]
Both of these points of view are useful for proving certain results, but we will not adopt them in the present paper.

The ODE has the following conserved quantities, which are easy to check by direct computation:
\begin{enumerate}
\item $X=\sum m_i x_i$. This is actually two conserved quantities (the $x$ and $y$ coordinates). When the total mass $\sum m_i$ is nonzero, we can divide $X$ by it to get conservation of center of mass.
\item $I=\sum m_i |x_i|^2$ is the second moment (physically, it corresponds to the fluid's angular momentum).
\item $\mathcal E=\sum_{1\le i<j\le n} m_im_j\log|x_i-x_j|$ is the energy.
\item $\tilde I=\sum_{1\le i<j\le n} m_im_j|x_i-x_j|^2$, obtained from an appropriate linear combination of $I$ and $|X|^2$. 
\end{enumerate}

It is easy to see that one-vortex systems are stationary and that two-vortex systems have circular or linear orbits. For three vortices the system is completely integrable (recall the Hamiltonian formulation) and its trajectory has been analyzed in \cite{Aref79}. For four vortices, the special case where the sum of masses and $X$ are both 0 also gives an integrable system \cite{Eckhardt88}.

Also using the Hamiltonian formulation, \cite{Khanin82} showed using KAM theory that for generic collections of masses, there is a positive measure set of initial data that leads to quasiperiodic motion, constructed hierarchically with nearby pairs of vortices in high-frequency orbits, such pairs of vortices orbiting each other at a much larger distance with much larger period, and so on.

There have also been a large number of results obtaining relative equilibria, that is configurations which are stationary, rotate, or translate uniformly, and analyzing their behavior (see \cite{VortexCrystals03} and \cite{Newton14} for surveys of results in this direction). The only result we'll state along these lines is that three vortices at the vertices of an equilateral triangle, whatever their masses, give a relative equilibrium \cite{Novikov75}. There are also self-similarly expanding or contracting configurations, which we discuss in Section~\ref{sec:examples}.

Finally, there is the question of rigorously justifying the point vortex model, that is obtaining solutions to the ODE as limits of smooth solutions to the 2D Euler equation by replacing each point vortex with a localized smooth nonpositive or nonnegative vorticity function. This justification is obtained for different assumptions on the vorticity function in~\cite{MarchioroPulvurenti93}, \cite{Marchioro98}, \cite{Serfati98}, \cite{OtherSerfati98} (in order of weakening assumptions and stronger conclusions).

We now turn to a discussion of when we can tell based on masses of point vortices that we have global existence and boundedness of solutions. For one, two, or three vortices, this can be seen when Conditions~\ref{cond:sumnonzero} and \ref{cond:nospiral} below are satisfied by using complete integrability mentioned above. For more than 3 vortices, we note that by standard ODE theory, a point vortex system can fail to have global existence only if
\[
\liminf_{s\to T^-} \min_{1\le i<j\le k} |x_i(s)-x_j(s)|=0.
\]
Following an argument found in \cite{MarchioroPulvurenti94pointchapter}, we note that when all the masses have the same sign, conservation of $I$ gives us an upper bound on the solution while it exists, and then conservation of $\mathcal E$ gives us a lower bound on the distance between vortices for all time, both easily computed from the initial data. We aim to obtain results in this vein when the masses can have different signs.

Outside of Section~\ref{sec:examples}, we typically assume that the $n$ point vortices satisfy the following conditions:
\begin{condition}[No Translation Condition]\label{cond:sumnonzero}
No nonempty subset has total mass 0.
\end{condition}
and
\begin{condition}[No Spiral Condition]\label{cond:nospiral}
Take any partition
\[
S=S_0\sqcup S_1\sqcup\cdots\sqcup S_k
\]
of the indices into sets where $k\ge 2$ and all the sets with the possible exception of $S_0$ are nonempty. If we let
\[
\Omega_i=\sum_{j\in S_i} m_j
\]
then we have
\[
\sum_{1\le i<j\le k} \Omega_i\Omega_j\ne 0.
\]
\end{condition}
In Section~\ref{sec:examples}, we explain why we impose these conditions and give some examples of bad behavior that can occur if they are not satisfied.

The only previous results in these directions that we are aware of for more than three vortices of varying sign and which don't depend on particularities of the initial configuration come from  \cite{MarchioroPulvurenti94pointchapter} and are as follows:
\begin{enumerate}
\item Assume the No Translation Condition (Condition~\ref{cond:sumnonzero}) is satisfied. Then over a fixed time interval, there is a bound (depending on the masses) for how much the position of any vortex can change.
\item Assume the No Translation Condition is satisfied. Then for Lebesgue-almost all initial data, there is global existence. For a fixed time interval, they also quantify the measure of initial data which will lead to two vortices coming within distance $\epsilon$ of each other.
\end{enumerate}
We prove the following results, stated here informally. Note that if we reverse signs of masses and reflect the system over the $x$ axis, we will get identical evolution of the system, so all statements below have corresponding statements with signs reversed.
\begin{theorem}
\begin{enumerate}
\item  Assume the No Translation Condition (Condition~\ref{cond:sumnonzero}) is satisfied. Then the position of vortices can change by at most $C\sqrt{t}$ (see Theorem~\ref{thm:movementbound} for the precise statement).
\item Assume the No Translation Condition is satisfied. Then the only way the ODE can break down is by several vortices colliding at a point (see the discussion at the end of Section~\ref{sec:movementbound}).
\item Assume that all the masses are positive, except one that is negative and bigger in absolute value than all the others combined. Then the solution has global existence and is bounded (see Theorem~\ref{thm:onebignegative} for the precise statement).
\item For certain other open sets of masses, we have global existence and growth slower than $\sqrt{t}$. See Theorem~\ref{thm:nocollisionsandslowgrowth} and Section~\ref{sec:comments} for a discussion of what open sets we obtain this for and what bounds we obtain.
\end{enumerate}
\end{theorem}
Based on the methods used in the proof, one might hope that the following more general result holds:
\begin{conjecture}\label{conj:weak}
Suppose that the masses satisfy the No Translation and No Spiral Conditions. Then any solution has global existence and grows as $o(\sqrt{t})$.
\end{conjecture}
Probably a somewhat stronger, more narrowly tailored to avoiding specific bad behaviors, version of this conjecture would also hold.
\subsection{Bounds on the kernel and notation}
Later on, we will need the following bounds on the kernel $K$:
\begin{equation}\label{Taylorbounds}
\begin{aligned}
|DK(x,y)|&\le \frac{C}{|x-y|^2}\\
|D^2K(x,y)|&\le \frac{C}{|x-y|^3}\\
|D^3K(x,y)|&\le \frac{C}{|x-y|^4}
\end{aligned}
\end{equation}

We will let $C=C(m_1,\ldots,m_n)$ be a big enough positive constant, possibly denoting a different constant on different lines. When we use $O$ notation, the implicit constant will be allowed to depend on all the masses.

We will use $[n]$ to denote $\{1,\ldots,n\}$.

\subsection{Acknowledgements}
The author wishes to express his gratitude to Theodore Drivas for discussions of related problems and for encouraging him to finish writing up the present work. This work was supported by the National Science Foundation Graduate Research Fellowship Program under Grant No. DGE-1656466.

\section{Examples}\label{sec:examples}
If the No Translation Condition (Condition~\ref{cond:sumnonzero}) is not satisfied, the simplest example of bad behavior is a pair of vortices with $m_1=-m_2$ translating at constant speed, in which case the solution grows linearly. We can also have some other vortices in the system, as long as the pair of vortices  with masses $\pm m$ do not come too close to other vortices (for which we need a suitably favorable starting configuration). Then they can separate out, with distance from them to any other vortices growing linearly, and it is not hard to check that in the limit, they will translate with speed $C+O(1/t)$. In addition, one can expect, though we do not in general know how to prove except in the cases of relative equilibria, similar behavior for systems with more than two vortices for which total vorticity is 0, but $X\ne 0$. However, if the No Translation Condition is satisfied, then conservation of the center of mass should rule out such behaviors. This is why we need to include the No Translation Condition.

If the No Spiral Condition is not satisfied, more specifically if we have
\[
\sum_{1\le i<j\le n} m_im_j=0,
\]
then there are numerous constructed examples of solutions which self-similarly spiral outwards, growing as $\sqrt{t}$. An analysis of such self-similarly evolving 3-vortex system can be found in \cite{Aref07}. Some self-similarly evolving 4 and 5 vortex systems are constructed and analyzed in \cite{NovikovSedov79}. Some numerics for self-similarly evolving systems with more vortices may be found in \cite{Kudela14}. By reflecting these solutions, or by reversing the signs of the vorticities, we can construct examples that self-similarly spiral inwards and collapse in finite time. However, if
\[
\sum_{1\le i<j\le n} m_im_j\ne 0,
\]
then such self-similarly spiraling solutions are impossible, as they would violate conservation of the energy $\mathcal E$.

It is also possible to add some vortices far away (and slightly perturb initial positions) and still obtain a collapse~\cite{GrottoPappalettera22}. Similarly, if we have a system spiraling outwards with distances between vortices growing like $\sqrt{t}$, it seems plausible that we can replace each vortex with a cluster of vortices having the same total mass, and still get growth of the system as $\sqrt{t}$; in fact, we will construct such solutions in Corollary~\ref{corr:spiral}. It is in order to preclude these possibilities that we introduce the No Spiral Condition. We note that to break these mechanisms, we can use slightly weaker conditions. 

\section{Movement bound} \label{sec:movementbound}
We have the following statement, which essentially says that  each vortex of a set can move by at most $C\sqrt{T}$ over time $T$, assuming they stay far enough away from other vortices. The proof is identical to a proof in \cite{MarchioroPulvurenti94pointchapter}, but with dependence on the time interval quantified. It is interesting as a result in its own right, but it is also a model for the proof of Theorem~\ref{thm:nocollisionsandslowgrowth}.
\begin{lemma}\label{lem:movementbound}
Suppose we have a finite set of masses $m_1,\ldots,m_n$ satisfying the No Translation Condition (Condition~\ref{cond:sumnonzero}). Then for any subset $S\subseteq [n]$, there exists a constant $C_S$ depending on all the masses so that for any $T$, if we have a solution of the ODE on time interval $[t_1,t_1+T)$ with
\begin{equation}\label{cond:isolated}
|x_i(s)-x_j(s)|>\sqrt{T}
\end{equation}
for all $i\in S,j\notin S$, $s\in [t_1,t_1+T]$ then for any $s\in [t_1,t_1+T]$ and $i\in S$, we have that $|x_i(s)-x_i(t_1)|<C_S\sqrt{T}$.
\end{lemma}
\begin{proof}
We induct on the set $S$. If $|S|=1$, then at any point  in time \eqref{cond:isolated} tells us that the velocity is at most $\frac{C}{\sqrt T}$, so the total distance traveled is at most $C\sqrt{T}$.

Assume the statement holds for all proper subsets of $S$. We choose $C_S$ sufficiently large, depending on the masses and on $C_A$ for all $A\subseteq S$, to make the following argument go through.

First, we note that the center of mass of $S$, denoted $x_S$, does not move from internal interactions of $S$, so it only moves due to the external field coming from vortices not in $S$. Due to \eqref{cond:isolated}, this external velocity field applies velocity at most $\frac{\tilde C}{\sqrt T}$ to each vortex in $S$, so it applies velocity at most $\frac{\hat C}{\sqrt T}$ to the center of mass $x_S$, so $x_S$ moves by at most $\hat C\sqrt{T}$. Now, suppose for the sake of contradiction that for some $i\in S$ and some time $s$, we have $|x_i(s)-x_i(t_1)|\ge C_S\sqrt{T}$. Then, by taking $C_S>3\hat C$, we know that either at time $s$ or at time $t_1$, we have
\[
|x_i-x_S|\ge \frac{C_S}{3}\sqrt{T}.
\]
Without loss of generality, this occurs at time $s$ (otherwise we simply use $s$ to denote the time $t_1$). Because the total mass of $S$ is not 0, this implies that at time $s$, we have that for some $j\in S$,
\[
|x_i-x_j|\ge \frac{\sum_{k\in S} m_k}{\sum_{k\in S} |m_k|}\frac{C_S}{3}\sqrt{T}.
\]
We now can split $S$ into two clusters that are sufficiently far apart. In the worst case, the vortices of $S$ are evenly spaced with $x_i$ and $x_j$ at the two endpoints. Whatever the arrangement, there must be some partition
\[
S=A\sqcup B
\]
with $A$, $B$ nonempty and
\begin{equation}\label{Ssplit}
|x_k(s)-x_\ell(s)|>\frac{\sum_{k\in S} m_k}{\sum_{k\in S} |m_k|}\frac{C_S}{3(|S|-1)}\sqrt{T}
\end{equation}
for all $k\in A,\ell\in B$. We now just need that the coefficient in front of $\sqrt{T}$ in \eqref{Ssplit} is greater than $C_A+C_B+1$ for all possible partitions $S=A\sqcup B$, which is achievable by taking $C_S$ sufficiently large. We then obtain a contradiction and complete the proof by using the inductive hypothesis and making the bootstrap assumption that clusters $A$ and $B$ maintain a spatial separation of $\sqrt{T}$ for the entire time interval $[t_1,t_1+T]$ (all the times of which are at time interval at most $T$ forward or backward from $s$).
\end{proof}

We can now apply the lemma with $S=[n]$ to obtain
\begin{theorem}\label{thm:movementbound}
Suppose we have a finite set of masses $m_1,\ldots,m_n$ satisfying the No Translation Condition. Then there exists some constant $C=C(m_1,\ldots,m_n)$ so that  for any $T$, assuming that we have a solution of the ODE on time interval $[t_1,t_1+T)$, we have that $|x_i(s)-x_i(t_1)|<C\sqrt{T}$.
\end{theorem}
We make the observation that if we want to get a better constant $C$ in the theorem statement, we can abandon the induction structure of the proof, and instead solve an optimization problem with various constants attached to various clusters $S\subseteq [n]$. While doing this is not particularly worthwhile, we note it here because we will make an analogous  observation in a more interesting context after the proof of Theorem~\ref{thm:nocollisionsandslowgrowth}.

We can also use Theorem~\ref{thm:movementbound} to classify ways in which solutions can break down when the No Translation Condition is satisfied.

Suppose that the ODE has a solution on time $[0,T)$ that does not continue past $T$. It is easy to see from standard ODE existence and uniqueness that the only way this happens is if
\[
\liminf_{s\to T^-} \min_{1\le i<j\le k} |x_i(s)-x_j(s)|=0.
\]
From this it follows that there is at least one pair $(i,j)$ satisfying $1\le i<j\le k$ with
\[
\liminf_{s\to T^-} |x_i(s)-x_j(s)|=0.
\]
Applying Theorem~\ref{thm:movementbound} over intervals $[T-\epsilon,T)$, we obtain that
\[
\lim_{s\to T^-} |x_i(s)-x_j(s)|=0.
\]
If this holds, we say that we have a \emph{collision} of vortices $i$ and $j$. It is easy to see that collision is an equivalence relation, so at time $T$, we have one or several collisions of several vortices each, and the distance between a given collision and other vortices has some positive lower bound on $[T-\epsilon,T)$. Note that since we use Theorem~\ref{thm:movementbound}, we need the No Translation Condition for this argument. In fact, it may be possible to construct a blowup where a pair of vortices with mass $\pm 1$ moves around faster and faster, moving closer and thus getting faster each time it scatters off of another vortex, and has no limiting position.

\section{One big negative}\label{sec:negative}
\begin{theorem}\label{thm:onebignegative}
Suppose that $m_2,\ldots m_n>0$ and $m_1<0$ with $-m_1>m_2+\cdots+m_n$. Then the ODE solution exists and is bounded for all time.
\end{theorem}
\begin{proof}
First, we prove existence. Suppose for the sake of contradiction that we have a collision at time $T$, and that $S\subset [n]$ is the set of vortices involved in that collision (since the No Translation Condition holds, this is the only way that breakdown can occur by the discussion at the end of Section~\ref{sec:movementbound}). We let $R>0$ and $\delta>0$ be chosen such that on the time interval $[T-\delta,T]$, the distance between any vortex in the collision and any vortex outside the collision is lower bounded by $R$ and the distance between any two of the colliding vortices is upper bounded by $R/100$. Then we have that
\[
\tilde I_S=\sum_{\substack{i,j\in S\\i<j}} m_im_j|x_i-x_j|^2
\]
is an almost-conserved quantity, in the sense that it doesn't change due to interactions of the vortices in the collision. Then we use the bound on $DK$ in \eqref{Taylorbounds} to get that the velocity difference between vortices $i,j\in S$ arising from the outside sources is at most
\[
C\frac{|x_i-x_j|}{R^2},
\]
so
\begin{equation}\label{ISODE}
\left|\frac{d}{dt} \tilde I_S\right|\le \sum_{\substack{i,j\in S\\i<j}} C\frac{|x_i-x_j|^2}{R^2}
\end{equation}

We now have two cases.

If $1\notin S$, then we immediately get first, that $\tilde I_S$ is positive shortly before time $T$, and second, that the right-hand side of $\eqref{ISODE}$ is bounded by $C\tilde I_S/R^2$, so
\[
\left|\frac{d}{dt} \tilde I_S\right|\le \frac{C}{R^2}\tilde I_S
\]
so $\tilde I_S$ cannot go to 0 as $t\to T^{-}$, which contradicts the vortices in $S$ colliding.

Now suppose $1\in S$. We will assume wlog that $S=[k]$ and we will translate our coordinates in a time-dependent way so that $x_1(t)=0$ (this does not change any of the values $\tilde I_S$, the right-hand side of \eqref{ISODE}, or the two quantities we are comparing below). Then we have that
\begin{align*}
\sum_{2\le i<j\le k}m_im_j |x_i-x_j|^2&=\left(\sum_{i=2}^k m_i\right)\sum_{i=2}^k m_i|x_i|^2-\sum_{i=2}^k\sum_{j=2}^k m_im_jx_i\cdot x_j\\
&=\left(\sum_{i=2}^k m_i\right)\sum_{i=2}^k m_i|x_i|^2-\left|\sum_{i=2}^k m_ix_i\right|^2\\
&\le \left(\sum_{i=2}^k m_i\right)\sum_{i=2}^k m_i|x_i|^2\\
&\le \frac{\sum_{i=2}^k m_i}{-m_1}\left(-m_1\sum_{i=2}^k m_i|x_i-x_1|^2\right).
\end{align*}
Thus, if we look at both the positive and the negative terms of $\tilde I_S$ (where the positive terms do not involve $m_1$ and the negative ones do), we see that the absolute value of the sum of the negative terms is bigger than the sum of the positive ones by a factor of at least
\[
1+\epsilon=\frac{-m_1}{\sum_{i=2}^n m_i}.
\]
From this we get that
\[
\sum_{1\le i<j\le n}^n |x_i-x_1|^2\le -C\tilde I_S.
\]
From this we get that $\tilde I_S$ is negative shortly before time $T$. Also, combining this with $\eqref{ISODE}$, we get that
\[
\left|\frac{d}{dt} \tilde I_S\right|\le \frac{-C}{R^2}\tilde I_S
\]
so $\tilde I_S$ cannot go to 0 as $t\to T^-$, which contradicts the vortices in $S$ colliding. This finishes the proof of global existence and uniqueness of solutions to the ODE.

We now prove boundedness, using conservation of the center of mass and second moment. Without loss of generality, we can assume that the center of mass of the entire system is at 0. Then, for any $t>0$, we have that the center of mass of vortices $2,\ldots,n$ is
\[
\frac{-m_1}{m_2+\cdots+m_n}x_1(t)
\]
so the second moment is
\begin{align*}
I&=\sum_{i=1}^n m_i|x_i(t)|^2\\
&\ge m_1|x_1(t)|^2+\left(\sum_{i=2}^n m_i\right)\left|\frac{-m_1}{m_2+\cdots+m_n}x_1(t)\right|^2\\
&\ge m_1|x_1(t)|^2+\frac{m_1^2}{m_2+\cdots+m_n}|x_1(t)|^2\\
&\ge c(m_1,\ldots,m_n)|x_1(t)|^2
\end{align*}
Thus $x_1$ is bounded. Since all the other masses are positive, the conservation of the second moment then gives us that the entire system is bounded.
\end{proof}

\section{Clusters}
For our proof of Theorem~\ref{thm:nocollisionsandslowgrowth}, we will need to define clusters and prove a lemma about them, which we do in this section.
\begin{definition}
Nonempty $S\subset  [n]$ is an $a$-cluster at time $t$ if $|x_i(t)-x_j(t)|\le a$ for all $i,j\in S$. An $a$-cluster is $b$-isolated if  $|x_i(t)-x_j(t)|\ge b$ for all $i\in S, j\notin S$.
\end{definition}

Some notes about this definition:
\begin{itemize}
\item For any $1\le i\le n$, we have that $\{i\}$ is a 0-cluster.
\item $[n]$ is $\infty$-isolated.
\item If we talk about a $b$-isolated $a$-cluster, we are implicitly assuming that $a\ll b$ (more precisely, there is some large constant $C$ depending on all the masses and we will always have $Ca\le b$). The constant $C$ is necessary since we will need to use the centers of mass of the clusters and, depending on the masses and their signs, we can have for instance that the center of mass of an $a$-cluster is at distance $100a$ from any of its vortices.
\end{itemize}

When we talk about the distance between two clusters $d(S_1,S_2)$, we will mean the distance between their centers of mass, but since we will only care about this distance approximately, and since the clusters will be sufficiently isolated, we could equally well take the closest distance of a vortex in one to a vortex in the other, the furthest distance of a vortex in one to a vortex in the other, or any other similar definition.

Now suppose we have a $b$-isolated $a$-cluster $S$. We can let
\[
m_S=\sum_{i\in S} m_i
\]
be the total mass of $S$ and
\[
x_S=\frac{\sum_{i\in S} m_ix_i}{m_S}
\]
be the center of mass of $S$ and consider the (negative) change in energy we would get if we replaced the cluster with a single vortex of the mass $m_S$ located at $x_S$ (the center of mass of $S$). Specifically, we recall that
\[
 \mathcal E=\sum_{1\le i<j\le n} m_im_j\log|x_i-x_j|
\]
is the energy. We let $\overline E_S$ be the energy of the vortex configuration where we replace the cluster $S$ with a single vortex of mass $m_S$  located at $x_S$, that is
\[
\overline E_S=\sum_{\substack{i,j\notin S\\i<j}}m_im_j\log|x_i-x_j|+\sum_{i\notin S}m_im_S\log|x_S-x_i|.
\]
We now define
\begin{equation}\label{eq:mathcalES}
\mathcal E_S=\mathcal E-\overline E_S=\sum_{i\notin S,j\in S}m_im_j\left(\log|x_i-x_j|-\log|x_S-x_i|\right)+\sum_{\substack{i,j\in S\\i<j}}m_im_j\log|x_i-x_j|.
\end{equation}
This is approximately equal to just the interaction energy within the cluster (the second sum on the right-hand-side), but we want this exact form of the energy, since it allows us to prove the following lemma:
\begin{lemma}\label{lem:Ederiv}
If $S$ is a $b$-isolated $a$-cluster, then
\[
\left|\frac{d}{dt} \mathcal E_S\right|\le C\frac{a^2}{b^4}
\]
\end{lemma}
\begin{proof}
The proof is by Taylor expansion of the kernel $K$ and obtaining cancellations. 
We recall that $\mathcal E_S=\mathcal E-\overline E_S$. Note that $\mathcal E$ is conserved, so we need only understand how $\overline E_S$ changes. For simplicity of notation, we renumber the vortices so that $S=[k]$. Then we let $v_i,v_S$ be the velocities that would arise from the vortex configuration where we replaced the cluster $S$ with a single vortex of mass $m_S$ located at $x_S$. Namely for $i>k$, we have
\begin{align*}
v_i&=m_S\frac{(x_i-x_S)^\perp}{|x_i-x_S|^2}+\sum_{\substack{k<j\le n\\j\ne i}}m_j \frac{(x_i-x_j)^\perp}{|x_i-x_j|^2}\\
v_S&=\sum_{k<j\le n}m_j \frac{(x_S-x_j)^\perp}{|x_S-x_j|^2}
\end{align*}
and let $u_i, u_j,u_S$ be the adjustments to the velocity so that 
\begin{align*}
\frac{d}{dt}x_i&=v_i+u_i\\
\frac{d}{dt}x_S&=v_S+u_S
\end{align*}
Then
\begin{align}
-\frac{d}{dt} \mathcal E_S=\frac{d}{dt} \overline E_S&=\sum_{k<i<j\le n} m_im_j\frac{\frac{d}{dt}(x_i-x_j)}{|x_i-x_j|^2}\cdot (x_i-x_j)+\sum_{k<i\le n} m_im_S\frac{\frac{d}{dt}(x_i-x_S)}{|x_i-x_S|^2}\cdot (x_i-x_S)\nonumber\\
&=\sum_{k<i<j\le n} m_im_j\frac{v_i-v_j}{|x_i-x_j|^2}\cdot (x_i-x_j)+\sum_{k<i\le n} m_im_S\frac{v_i-v_S}{|x_i-x_S|^2}\cdot (x_i-x_S)\nonumber\\
&\qquad+\sum_{k<i<j\le n} m_im_j\frac{u_i-u_j}{|x_i-x_j|^2}\cdot (x_i-x_j)+\sum_{k<i\le n} m_im_S\frac{u_i-u_S}{|x_i-x_S|^2}\cdot (x_i-x_S)\nonumber\\
&=\sum_{k<i<j\le n} m_im_j\frac{u_i-u_j}{|x_i-x_j|^2}\cdot (x_i-x_j)+\sum_{k<i\le n} m_im_S\frac{u_i-u_S}{|x_i-x_S|^2}\cdot (x_i-x_S)\label{ddtoverlineES}
\end{align}
where all the terms involving $v_i$ or $v_S$ cancel by the same calculation that gives conservation of energy for a point vortex system. To bound the right-hand side of \eqref{ddtoverlineES}, we start by bounding $u_i$ and $u_S$.
We note that by the same calculation which gives conservation of the center of mass of a point vortex system, the interactions within the cluster have no net effect on the derivative of $x_S$, so
\[
u_S=\frac{d}{dt}x_S-v_S=\sum_{j=1}^k\sum_{i=k+1}^n m_i\frac{m_j}{m_S}\left(\frac{(x_j-x_i)^\perp}{|x_j-x_i|^2}-\frac{(x_S-x_i)^\perp}{|x_S-x_i|^2}\right).
\]
We now Taylor expand $K(x,y)=\frac{(x-y)^\perp}{|x-y|^2}$ around $x=x_S$, see that the linear terms cancel by the definition of the center of mass, and use \eqref{Taylorbounds} to get that
\begin{equation}\label{uSbound}
|u_S|\le C\frac{a^2}{b^3}.
\end{equation}
We now bound $u_i$ for $i>k$ by calculating
\[
u_i=\frac{d}{dt}x_i-v_i=\sum_{j=1}^k m_j\left(\frac{(x_i-x_j)^\perp}{|x_i-x_j|^2}-\frac{(x_i-x_S)^\perp}{|x_i-x_S|^2}\right).
\]
We now Taylor expand $K(x,y)=\frac{(x-y)^\perp}{|x-y|^2}$ around $y=x_S$, see that the linear terms cancel by the definition of the center of mass, and use \eqref{Taylorbounds} to get that
\begin{equation}\label{uibound}
|u_i|\le C\frac{a^2}{b^3}.
\end{equation}
Together, \eqref{uibound} and \eqref{uSbound} give a bound of $C\frac{a^2}{b^4}$ for all terms on the right-hand side of \eqref{ddtoverlineES} except those where $|x_i-x_j|\le b/10$. For those terms, we get
\begin{align*}
u_i-u_j&= \frac{d}{dt}x_i-v_i-\left(\frac{d}{dt}x_j-v_j\right)\\
&=\sum_{\ell=1}^k m_\ell\left(\frac{(x_i-x_\ell)^\perp}{|x_i-x_\ell|^2}-\frac{(x_i-x_S)^\perp}{|x_i-x_S|^2}\right)-m_\ell\left(\frac{(x_j-x_\ell)^\perp}{|x_j-x_\ell|^2}-\frac{(x_j-x_S)^\perp}{|x_j-x_S|^2}\right).
\end{align*}
We now Taylor expand $K(x,y)=\frac{(x-y)^\perp}{|x-y|^2}$ around $x=x_i, y=x_S$, and see that the first term that doesn't cancel comes from $D_xD^2_y f$, which from \eqref{Taylorbounds} gives us that
\begin{equation}\label{uiujbound}
|u_i-u_j|\le C\frac{a^2|x_i-x_j|}{b^4}.
\end{equation}
We now plug \eqref{uSbound}, \eqref{uibound}, and \eqref{uiujbound} into \eqref{ddtoverlineES} to obtain the lemma statement.
\end{proof}

\section{No collisions and growth slower than $T^{1/2}$}
Let  $\kappa=\kappa(n)$ be a sufficiently large constant, and define the following condition:
\begin{condition}\label{simplecondition}
For some $0\le h\le n$, we have that $m_1,\ldots,m_h$ are negative, $m_{h+1},\ldots,m_n$ are positive, and for $1\le i\le h<j\le n$, we have $\kappa|m_i|<m_j$.
\end{condition}
\begin{theorem}\label{thm:nocollisionsandslowgrowth}
If $m_1,\ldots,m_n$ satisfy Condition~\ref{simplecondition} then we get no collisions and growth of the ODE solution as $O(T^{1/2-\epsilon(n)})$.
\end{theorem}
We will be clear where in the proof we use Condition~\ref{simplecondition}.

The proof of this theorem is inspired by the proof of Theorem~\ref{thm:movementbound}, except that we use a different notion instead of position (the shifted logarithmic distance $\overline d$ between vortices defined below) and a different notion instead of center of mass (the average distance $\hat d$ defined below). Analogously to the way in which conservation of $X$ together with the No Translation Condition cause spatial separation if any vortex travels too far, we have that  conservation of the energy $\mathcal E$ together with the No Spiral Condition (Condition~\ref{cond:nospiral}) causes separation in scales if the shifted logarithmic distance between some two vortices changes by a large factor.

Analogously to Lemma~\ref{lem:movementbound}, we have the following lemma from which the theorem immediately follows:

\begin{lemma}\label{lem:nocollisionsandslowgrowth}
There exist constants $\alpha(k)$ for $2\le k\le n$ satisfying the following. Suppose we have a finite set of masses $m_1,\ldots,m_n$ satisfying Condition~\ref{simplecondition}. Suppose further we have a cluster $S\subseteq [n]$, and a partition into subclusters $S=S_1\sqcup\cdots\sqcup S_k$ with $k\ge 2$, and that over a time interval $[t_1,t_2]$ with $t_2-t_1\le T$ we have the following bootstrap assumptions:
\begin{enumerate}
\item  \label{overalldistbound} all distances are less than $T^{1/2-\epsilon(n)}$
\item \label{cond:scaleseparation}
\begin{align*}
\text{each $S_i$ is a }b_i\text{-isolated }&\left(\frac{b_i^2}{\sqrt{T}}\right)\text{-cluster}\\
\text{$S$ is an }f\text{-isolated }&\left(\frac{f^2}{\sqrt{T}}\right)\text{-cluster}.
\end{align*}
\end{enumerate}
for some $b_i,f$ independent of time. Then, as long as $T$ is sufficiently large, if at times $s_1,s_2\in [t_1,t_2]$ and for some $1\le i<j\le k$, we have that
\[
d(S_i,S_j)=T^{1/2-p_1}\qquad\text{and}\qquad d(S_i,S_j)=T^{1/2-p_2}
\]
respectively, then
\begin{equation}\label{lemmaconclusion}
\frac{p_1}{p_{2}}<\alpha(k)
\end{equation}
\end{lemma}

\begin{figure}[h]
\centering
\includegraphics[width=150pt]{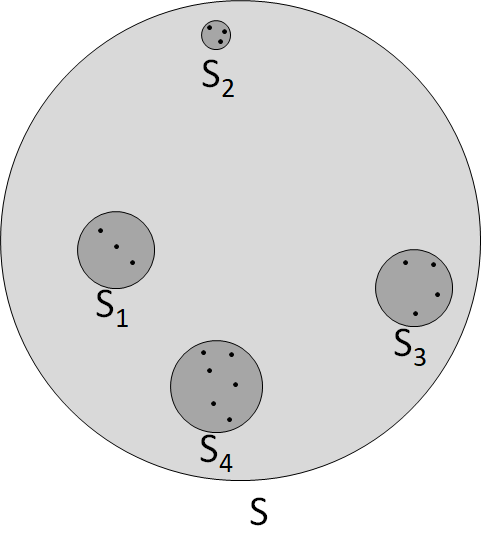}
\caption{This is an illustration of the clusters in the statement of Lemma~\ref{lem:nocollisionsandslowgrowth}}
\label{figg}
\end{figure}

In other words, if different scales stay sufficiently separated, as quantified by bootstrap assumption~\ref{cond:scaleseparation}, then the distance between two clusters cannot change too much, as quantified by \eqref{lemmaconclusion}.
\begin{proof}
We will prove this by induction on $k$, after introducing some definitions.

We will define an appropriately shifted logarithmic notion of distance:
\[
\overline d(S_i,S_j)=\frac{1}{2}-\frac{\log d(S_i,S_j)}{\log T}.
\]
That is, if $\overline d(S_i,S_j)=p$, then $d(S_i,S_j)=T^{1/2-p}$. Note that $\overline d(S_i,S_j)\le 0$ would correspond to a distance between vortices of $\ge \sqrt{T}$, and we will be proving that this is never achieved. Also, $\overline d$ is an inverted distance in the sense that it decreases when actual distance increases. Finally, note that the conclusion of the lemma we are proving is that $\overline d$ changes by at most a factor of $\alpha(k)$.

Now we define the \emph{interaction energy} of $S_1,S_2,\ldots,S_k$ via
\[
E(S_1,\ldots,S_k)=\mathcal E_S-\sum_{i=1}^k \mathcal E_{S_i}.
\]
The reason we use this precise expression is that Lemma~\ref{lem:Ederiv} gives bounds on how fast $E(S_1,\ldots,S_k)$ can change, but the reason we care about $E(S_1,\ldots,S_k)$ is that it is approximately equal to the energy of a point vortex configuration with $k$ point vortices, one at the center of mass of each cluster. More specifically, we can use \eqref{eq:mathcalES} to compute
\begin{align}
E(S_1,\ldots,S_k)&=\sum_{i=1}^k \sum_{\ell \in S_i}\sum_{j\notin S} m_im_j\big(\left(\log|x_\ell-x_j|-\log|x_S-x_j|\right)-\left(\log|x_\ell-x_j|-\log|x_{S_i}-x_j|\right)\big)\nonumber\\
&\qquad+\sum_{i=1}^k  \sum_{\substack{j,\ell\in S_i\\j<\ell}}m_jm_\ell\log|x_j-x_\ell|-m_jm_\ell\log|x_j-x_\ell|\nonumber\\
&\qquad+\sum_{1\le i<j\le k}\sum_{\ell_1\in S_{i}}\sum_{\ell_2\in S_{j}}m_{\ell_1}m_{\ell_2}\Big(\log|x_{\ell_1}-x_{\ell_2}|-(\log|x_{\ell_1}-x_{\ell_2}|-\log|x_{S_i}-x_{\ell_2}|)\nonumber\\
&\qquad\qquad\qquad-(\log|x_{\ell_2}-x_{\ell_1}|-\log|x_{S_j}-x_{\ell_1}|)\Big)\nonumber\\
&=\sum_{i=1}^k \sum_{\ell \in S_i}\sum_{j\notin S} O(1)+\sum_{1\le i<j\le k}\sum_{\ell_1\in S_{i}}\sum_{\ell_2\in S_{j}}m_{\ell_1}m_{\ell_2}(\log  d(S_i,S_j)+O(1))\nonumber\\
&=\sum_{1\le i<j\le k} m_{S_i}m_{S_j}\log d(S_i,S_j)+O(1)\label{eq:ES1Skapprox}.
\end{align}
Note that in the calculation above, we could have error bounds better than $O(1)$, but it suffices for our purposes.

We now define a notion of average distance between clusters as (essentially) the shifted logarithmic distance we would need between every pair of clusters to get the correct energy:
\[
\hat d(S_1,\ldots,S_k)=\frac{1}{2}-\frac{E(S_1,\ldots,S_k)}{(\log T)\sum_{1\le i<j\le k} m_{S_i}m_{S_j}}
\]
where the denominator is nonzero by the No Spiral Condition~(Condition~\ref{cond:nospiral}). Note that by $\eqref{eq:ES1Skapprox}$, we have
\begin{align}\label{eq:hatdisaverage}
\hat d=\frac{\sum_{1\le i<j\le k}m_{S_i}m_{S_j}\overline d(S_i,S_j)}{\sum_{1\le i<j\le k} m_{S_i}m_{S_j}}+O(1/\log T).
\end{align}
Note that while this is some sort of notion of ``average'' distance, there are signed weights involved, so it is possible for it to be larger or smaller than all the pairwise shifted logarithmic distances. 

We now turn to proving the lemma by induction on $k$. For $k=2$, we note that  bootstrap assumption~\ref{cond:scaleseparation} together with Lemma~\ref{lem:Ederiv} gives us that
\[
\left|\frac{d}{dt} E(S_1,S_2)\right|\le \left|\frac{d}{dt} \mathcal E_S\right|+\left|\frac{d}{dt} \mathcal E_{S_1}\right|+\left|\frac{d}{dt} \mathcal E_{S_2}\right|=O\left(\frac{1}{T}\right),
\]
so $E(S_1,S_2)$ changes by $O(1)$ over the relevant time interval. This means that $\hat d(S_1,S_2)$ changes by at most $O(1/\log T)$, so $\overline d(S_1,S_2)$ changes by $O(1/\log T)$. Since bootstrap assumption~\ref{overalldistbound} tells us that $\overline d(S_1,S_2)\ge \epsilon$, and since we can take $T$ sufficiently large, this finishes the proof for the case $k=2$.

Now suppose $k\ge 3$. Assume for the sake of contradiction that there are some $1\le i<j\le k$ so that $\overline d(S_i,S_j)$ varies by a factor of at least $\alpha(k)$ and take a minimal time interval on which such variation occurs. Bootstrap assumption~\ref{cond:scaleseparation} together with Lemma~\ref{lem:Ederiv} gives us that
\[
\left|\frac{d}{dt} E(S_1,\ldots,S_k)\right|\le\left|\frac{d}{dt} \mathcal E_S\right|+\sum_{i=1}^k\left|\frac{d}{dt} \mathcal E_{S_i}\right| =O\left(\frac{1}{T}\right),
\]
so $E(S_1,\ldots,S_k)$ changes by $O(1)$ over the relevant time interval. This means that $\hat d(S_1,\ldots,S_k)$ changes by at most $O(1/\log T)$. Now, by bootstrap assumption~\ref{overalldistbound} we know that $\overline d(S_i,S_j)$ will always be at least $\epsilon(n)$. Since we are assuming $\overline d(S_i,S_j)$  varies by a factor of at least $\alpha(k)$ throughout the time interval, we have there must be some time $s\in [t_1,t_2]$ with
\begin{equation}\label{claimcondition}
\frac{\hat d(S_1,\ldots,S_k)}{\overline d(S_i,S_j)}\ge \frac{\sqrt{\alpha(k)}}{2}\qquad\text{or}\qquad\frac{\hat d(S_1,\ldots,S_k)}{\overline d(S_i,S_j)}\le \frac{2}{\sqrt{\alpha(k)}}
\end{equation}
where we used bootstrap assumption~\ref{overalldistbound} to ensure that either $\hat d$ is small enough that the second part of \eqref{claimcondition} is satisfied, or $\hat d$ is big enough that the $O(1/\log T)$ change is small.

We can take $s$ to be the earliest time that \eqref{claimcondition} holds for some $i,j$. We will assume for now, to be proven later,
\begin{claim}\label{claim}
If we have the bootstrap assumptions of Lemma~\ref{lem:nocollisionsandslowgrowth}, as well as Condition~\ref{simplecondition} and \eqref{claimcondition}, then there exist some $1\le i_1<j_1\le k$ and $1\le i_2<j_2\le k$ with
\begin{equation}\label{eq:distantscales}
\frac{\overline d(S_{i_1},S_{j_1})}{\overline d(S_{i_2},S_{j_2})}\ge (3\alpha(k-1)^2)^{k^2}
\end{equation}
\end{claim}
We are now almost done. We take all pairwise distances of the form $\overline d(S_i,S_j)$ and list them in increasing order. From \eqref{eq:distantscales}, we know that two consecutive numbers in the list must differ by a factor of at least $3\alpha(k-1)^2$. Let $z$ be some number in this gap. Introduce an equivalence relation on $[k]$ via $i\sim j$ if $d(S_i,S_j)<z$. This equivalence relation gives us subclusters $U_1,\ldots,U_\ell$ for $\ell<k$ and with each $U_j$ having as subclusters various $S_j$, with fewer than $k$ subclusters per $U_j$ (see Figure~\ref{figg2} for an illustration).

\begin{figure}[h]
\centering
\includegraphics[width=220pt]{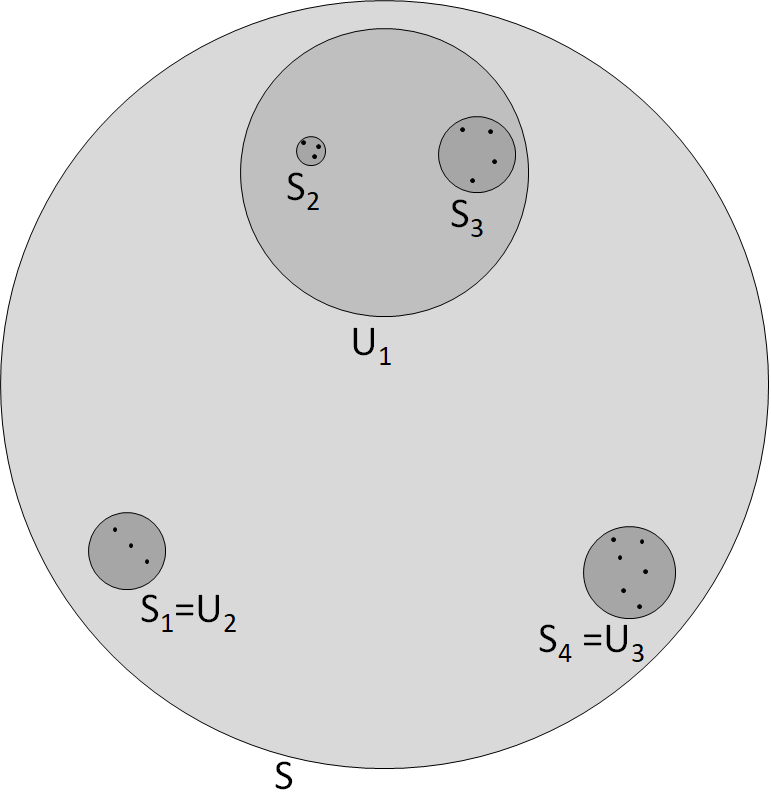}
\caption{This is an illustration of the clusters $U_i$}
\label{figg2}
\end{figure}
We now apply the inductive hypothesis both to each $U_j$ and to $S=U_1\sqcup\cdots\sqcup U_\ell$, going either forward or backward in time, to get a contradiction. The value $3\alpha(k-1)^2$ in the argument above was chosen so that the shifted logarithmic distances can change by a factor of $\alpha(k-1)$ without violating bootstrap assumption~\ref{cond:scaleseparation}.

All that remains is the proof of Claim~\ref{claim}.
\begin{proof}[Proof of Claim~\ref{claim}]
Suppose for the sake of contradiction that the conclusion of the claim is false. Then for all $\mu,\nu$, we have that
\[
\overline d(S_{\mu},S_{\nu})\le (3\alpha(k-1)^2)^{k^2}\overline d(S_i,S_j)
\]
Then \eqref{eq:hatdisaverage} and choosing $\alpha(k)$ to be sufficiently large compared to $\alpha(k-1)$ will give us that
\begin{align}
\hat d(S_1,\ldots,S_k)&\le\frac{\sum_{1\le \mu<\nu\le k}|m_{S_\mu}m_{S_\nu}|}{\left|\sum_{1\le \mu<\nu\le k} m_{S_\mu}m_{S_\nu}\right|}(3\alpha(k-1)^2)^{k^2}\overline d(S_i,S_j)+O(1/\log T)\nonumber\\
&< \frac{\sqrt{\alpha(k)}}{2}\overline d(S_i,S_j).\label{firstclaimresult}
\end{align}
To prove an inequality in the opposite direction, we note that \eqref{eq:hatdisaverage} involves a weighted average with weights $m_{S_\mu}m_{S_\nu}$. Let $\mu,\nu$ be the indices so that $m_{S_\mu}m_{S_\nu}$ is the largest weight in absolute value. Now Condition~\ref{simplecondition} gives us several cases (and this is the only place this condition shows up). The first case is that $m_{S_\mu},m_{S_\nu}$ are both negative. Then all the weights are positive. The second case is that $m_{S_\mu}>0$ and $m_{S_\nu}<0$. Then the largest weight by absolute value is negative, and all positive weights are smaller than it by a factor of at least $\kappa$. The last possibility is that $m_{S_\mu},m_{S_\nu}>0$. Then the largest weight by absolute value is positive, and all negative weights are smaller than it by a factor of at least $\kappa$. We will phrase the argument below assuming that $m_{S_\mu}m_{S_\nu}>0$. This handles the first and third cases, and the second case is handled by the same argument with a few sign reversals. Since we are assuming that the conclusion of Claim~\ref{claim} is false, we have that
\begin{equation}\label{SijSmunu}
\overline d(S_{i},S_{j})\le (3\alpha(k-1)^2)^{k^2}\overline d(S_{\mu},S_{\nu})
\end{equation}
so \eqref{eq:hatdisaverage} will give us that
\begin{align}
\hat d(S_1,\ldots,S_k)&\ge\frac{m_{S_\mu}m_{S_\nu}-k^2(3\alpha(k-1)^2)^{k^2}\frac{m_{S_\mu}m_{S_\nu}}{\kappa}}{k^2 m_{S_\mu}m_{S_\nu}}\overline d(S_\mu,S_\nu)+O(1/\log T)\nonumber\\
&\ge \frac{1}{2k^2}\overline d(S_\mu,S_\nu)
\end{align}
where we are using that $\kappa$ is sufficiently large compares to $\alpha(k-1)$ and that $T$ is sufficiently large. We now use \eqref{SijSmunu} and the fact that $\alpha(k)$ is sufficiently large compared to $\alpha(k-1)$ to obtain
\begin{equation}\label{secondclaimresult}
\hat d(S_1,\ldots,S_k)>\frac{2}{\sqrt{\alpha(k)}} \overline d(S_{i},S_{j}).
\end{equation}
Together, \eqref{firstclaimresult} and \eqref{secondclaimresult} combine to contradict the hypotheses of the claim, thus finishing its proof and the proof of Lemma~\ref{lem:nocollisionsandslowgrowth}.
\end{proof}
\end{proof}
To obtain the proof of Theorem~\ref{thm:nocollisionsandslowgrowth}, we now apply the Lemma~\ref{lem:nocollisionsandslowgrowth} to $S=[n]$ with $S_i=\{i\}$ for each $i$. Since all distances are initially of order 1, this gives that at sufficiently late time $T$, all distances are in the range
\[
\left[cT^{-\frac{\alpha(n)}{2}},CT^{\frac{1}{2}-\frac{1}{2\alpha(n)}}\right].
\]
This precludes the possibility of collisions. Also, now that we have an upper bound on distances between vortices, conservation of center of mass gives the same bound (up to a constant factor) for their positions.
\section{Improvements, corollaries,  and comments}\label{sec:comments}
In this section, we discuss potential improvements to Lemma~\ref{lem:Ederiv}, Lemma~\ref{lem:nocollisionsandslowgrowth}, and Theorem~\ref{thm:nocollisionsandslowgrowth}, as well as some corollaries. We start with improvements to the proofs:
\begin{enumerate}
\item
Condition~\ref{simplecondition} is not the most general condition under which the argument in the proof of Lemma~\ref{lem:nocollisionsandslowgrowth} works. For instance, it will work provided the following condition:
\begin{condition}\label{morecomplexcondition}
If we list the masses in increasing order by absolute value, then whenever $m_i$ and $m_{i+1}$ have opposite signs, we have $\kappa(n)|m_i|<|m_{i+1}|$.
\end{condition}
We can also have pairs of similar size and opposite signs, under some conditions that are annoying to phrase correctly. The important thing is that the No Translation and No Spiral Conditions hold and that the appropriate step in the proof of Claim~\ref{claim} goes through.
\item Given a specific list of masses, if we want to try to push the argument through in some more sensitive cases, we may need to forget about the convenience of doing an induction proof and using Lemma~\ref{lem:nocollisionsandslowgrowth}, and instead do a lot of casework based on what clusters we have and the relative scales, along with optimization of various parameters. We also need to do such an optimization if we want to get better upper bounds on growth of the system. The remaining comments in this list give some ideas we can exploit when we do such an optimization.
\item The proof as written never used the triangle inequality; using it can give improvements in constants, since it constrains possibilities for the distances $\overline d$.
\item The conservation of $\tilde I$ gives information which we can use about the behavior of the top-level cluster if all the masses of its subclusters have the same sign (or if one has a different sign from the rest and a bigger mass than the sum of the rest, see the arguments in Section~\ref{sec:negative}).
\item
Lemma~\ref{lem:Ederiv} is sufficient for the results as stated in this paper, but it's potentially possible to improve this lemma in some time-averaged sense when a cluster consists of only two subclusters. This is because the centers of mass of those two subclusters approximately follow circular orbits over appropriate timescales. This makes the behavior of the system isotropic in an appropriate time-averaged sense, leading to higher-order cancellation for most of the terms of \eqref{ddtoverlineES}. This doesn't deal with the $u_i-u_j$ terms in \eqref{ddtoverlineES} when $|x_i-x_j|$ is small, but those may be possible to deal with by grouping the other vortices into clusters and saying either they are far enough apart (say $\sqrt{ab}$) to get improvement, or we can group them into a cluster. When a cluster consists of three subclusters, it is a completely integrable system with some introduced forcing, and it's going to be isotropic on average unless we are exceptionally unlucky, so maybe some improvements along similar lines can be had there too.
\end{enumerate}
Next, we will give some direct corollaries of  Lemma~\ref{lem:nocollisionsandslowgrowth}. We formulate them in terms of  Condition~\ref{simplecondition}, but all the earlier considerations about what conditions we actually need apply.
\begin{corollary}\label{corr:fintime}
Suppose we have a point vortex system with masses $m_i$ and trajectories $x_i(t)$ evolving over a fixed time interval $[0,T]$, with some lower bound on the distance between them. Suppose that for each of the three vortices, we have a list of masses $m_i^1,\ldots,m_i^{n_i}$ summing to $m_i$ and satisfying Condition~\ref{simplecondition}. Then at time 0, we can replace each of the point vortices with a cluster of point vortices having masses $m_i^1,\ldots,m_i^{n_i}$, and all being within distance $\delta$ of $x_i(0)$. Then for the entire time interval, the vortices of cluster $i$ are within distance $f(\delta)$ of $x_i(t)$ where $f(\delta)\to 0$ as $\delta\to 0$. Here, the function $f$ is allowed to depend on all the masses of all the vortices in each swarm and on the trajectories $x_i(t)$.
\end{corollary}
Note that a weaker version of this result, where all point vortices in each cluster have the same sign, is proven in \cite{MarchioroPulvurenti94pointchapter}.

We now mention the result \cite{Zbarsky20threepatches} which allows one to take a self-similarly expanding system of 3 vortices and replace each point vortex with a vortex patch and still get a system where the distance between the patches grows as $\sqrt{t}$. From the point of view of stability of the 3-vortex system, the relevant fact about the evolution of the vortex patches was that the diameter of each vortex patch grows at most as $t^a$ for some $a<1/2$, which we can still obtain if instead of a vortex patch, we have a cluster of point vortices. To prove this corollary, one needs to combine the stability analysis for the 3-vortex system from \cite{Zbarsky20threepatches} with the bootstrap assumption that the diameter of each cluster is less than $\epsilon t^a$, which we get from Lemma~\ref{lem:nocollisionsandslowgrowth}.
\begin{corollary}\label{corr:spiral}
Suppose we have a self-similarly expanding system of three vortices with masses $m_1,m_2,m_3$ whose distance grows like $\sqrt{t}$. Suppose that for each of the three vortices, we have a list of masses $m_i^1,\ldots,m_i^{n_i}$ summing to $m_i$ and satisfying Condition~\ref{simplecondition}. Then for $\delta$ sufficiently small, we can replace each of the three point vortex with a cluster of vortices having diameter at most $\delta$ and masses $m_i^1,\ldots,m_i^{n_i}$ and then at all times $t$, we will have three clusters, each of diameter at most $t^{a}$ for some $a<1/2$ and with the distance between clusters being at least $c\sqrt{t}$.
\end{corollary}
Note that this includes the case where all vortices in the cluster have the same sign (since Condition~\ref{simplecondition} is then satisfied). Note also that if one obtained stability for some self-similarly expanding system of more than three point vortices, one would then obtain a result corresponding to Corollary~\ref{corr:spiral} for that system.

Finally, we note a limitation of the method of proof used for Lemma~\ref{lem:nocollisionsandslowgrowth}. One might hope that this proof technique could allow us to prove Conjecture~\ref{conj:weak}. This runs into the following difficulty:

Imagine we have $n=10$ and $S=\{6,7,8,9,10\}$. As $t$ approaches $T$, we set all distances inside the cluster $S$ to be $\sim (T-t)^{3/2}$ and all other distances to be $\sim T-t$. By appropriately choosing the masses, we can get that the total energy is conserved.  Then Lemma~\ref{lem:Ederiv} tells us that the internal energy of the cluster can change at rate at most
\[
\frac{C\left((T-t)^{3/2}\right)^2}{(T-t)^{4}}=\frac{C}{T-t},
\]
but this is compatible with the behavior described above and blowup at time $T$. To forbid such examples, we need some variant of Condition~\ref{simplecondition} (or new ideas).

\bibliographystyle{abbrv}
\bibliography{zbarskybib}

\begin{thebibliography}{10}

\bibitem{Aref79}
H.~Aref.
\newblock Motion of three vortices.
\newblock {\em The Physics of Fluids}, 22(3), 1979.

\bibitem{Aref07}
H.~Aref.
\newblock Point vortex dynamics: a classical mathematics playground.
\newblock {\em J. Math. Phys.}, 48(6):065401, 23, 2007.

\bibitem{VortexCrystals03}
H.~Aref, P.~Newton, M.~Stremler, T.~Tokieda, and D.~Vainchtein.
\newblock Vortex crystals.
\newblock {\em Advances in Applied Mechanics}, 39:1--79, 12 2003.

\bibitem{Eckhardt88}
B.~Eckhardt.
\newblock Integrable four vortex motion.
\newblock {\em Phys. Fluids}, 31(10):2796--2801, 1988.

\bibitem{GrottoPappalettera22}
F.~Grotto and U.~Pappalettera.
\newblock Burst of point vortices and non-uniqueness of 2{D} {E}uler equations.
\newblock {\em Archive for Rational Mechanics and Analysis}, 245:89–125,
  2022.

\bibitem{Khanin82}
K.~M. Khanin.
\newblock Quasi-periodic motions of vortex systems.
\newblock {\em Physica D: Nonlinear Phenomena}, 4(2):261--269, 1982.

\bibitem{Kudela14}
H.~Kudela.
\newblock Collapse of {$n$}-point vortices in self-similar motion.
\newblock {\em Fluid Dyn. Res.}, 46(3):031414, 16, 2014.

\bibitem{Marchioro98}
C.~Marchioro.
\newblock On the localization of the vortices.
\newblock {\em Boll. Unione Mat. Ital. Sez. B Artic. Ric. Mat. (8)},
  1(3):571--584, 1998.

\bibitem{MarchioroPulvurenti93}
C.~Marchioro and M.~Pulvirenti.
\newblock Vortices and localization in {E}uler flows.
\newblock {\em Comm. Math. Phys.}, 154(1):49--61, 1993.

\bibitem{MarchioroPulvurenti94pointchapter}
C.~Marchioro and M.~Pulvirenti.
\newblock {\em Mathematical theory of incompressible nonviscous fluids},
  volume~96 of {\em Applied Mathematical Sciences}, chapter~4, pages 134--177.
\newblock Springer-Verlag, New York, 1994.

\bibitem{Newton14}
P.~K. Newton.
\newblock Point vortex dynamics in the post-{A}ref era.
\newblock {\em Fluid Dyn. Res.}, 46(3):031401, 11, 2014.

\bibitem{Novikov75}
E.~Novikov.
\newblock Dynamics and statistics of a system of vortices.
\newblock {\em Zh. Eksper. Teoret. Fiz.}, 68, 1974.

\bibitem{NovikovSedov79}
E.~Novikov and Y.~Sedov.
\newblock Vortex collapse.
\newblock {\em J. Exp. Theor. Phys.}, (50):297–301, 1979.

\bibitem{Serfati98}
P.~Serfati.
\newblock Borne en temps des caract\'eristiques de l'\'equation d'{E}uler 2d
  \`a tourbillon positif et localisation pour le mod\`ele point-vortex, 1998.
\newblock Manuscript.

\bibitem{OtherSerfati98}
P.~Serfati.
\newblock Tourbillons-presque-mesures spatialement born\'{e}s et \'{e}quation
  d'{E}uler 2{D}, 1998.
\newblock Manuscript.

\bibitem{Zbarsky20threepatches}
S.~Zbarsky.
\newblock From point vortices to vortex patches in self-similar expanding
  configurations.
\newblock 2020.

\end{thebibliography}
\end{document}